\documentclass{daj}
\usepackage{amsmath}
\usepackage{amssymb}
\usepackage{amsfonts}
\usepackage{amsthm}
\usepackage{stmaryrd}
\usepackage{url}
\usepackage{hyperref}
\usepackage{fancyhdr}

\usepackage{amstext}
\usepackage{array}

\newtheorem{thm}{Theorem}[section]
\newtheorem{lemma}[thm]{Lemma}
\newtheorem{prop}[thm]{Proposition}

\theoremstyle{definition}

\newtheorem*{rem}{Remark}
\newtheorem{claim}{Claim}

\theoremstyle{definition}
\newtheorem{defn}[thm]{Definition}

\newcommand{\norm}[1]{{\left\|{#1}\right\|}}
\newcommand{\normoo}[1]{{\left\|{#1}\right\|}_{\infty}}

\newcommand{\set}[1]{{\left\{{#1}\right\}}}

\newcommand{\al}{\ensuremath{\alpha}}

\newcommand{\La}{\ensuremath{\Lambda}}
\newcommand{\la}{\ensuremath{\lambda}}
\newcommand{\sig}{\ensuremath{\sigma}}

\newcommand{\om}{\ensuremath{\omega}}
\newcommand{\Om}{\ensuremath{\Omega}}

\newcommand{\bears}{\begin{eqnarray*}}
\newcommand{\eears}{\end{eqnarray*}}

\newcommand{\ZZ}{\ensuremath{\mathbb{Z}}}

\newcommand{\QQ}{\ensuremath{\mathbb{Q}}}
\newcommand{\RR}{\ensuremath{\mathbb{R}}}
\newcommand{\NN}{\ensuremath{\mathbb{N}}}

\newcommand{\wh}[1]{\ensuremath{\widehat{#1}}}

\newcommand{\Gal}{\ensuremath{\textnormal{Gal}}}

\newcommand{\Php}{\ensuremath{\Phi_p(X^{N/p})}}
\newcommand{\Phq}{\ensuremath{\Phi_q(X^{N/q})}}

\newcommand{\sm}{\ensuremath{\setminus}}
\newcommand{\ssq}{\ensuremath{\subseteq}}

\newcommand{\vn}{\ensuremath{\varnothing}}

\newcommand{\pnq}{\ensuremath{p^nq}}

\newcommand{\pdiv}{\mid\!\mid}

\newcommand{\out}[1]{}

\numberwithin{equation}{section}

\dajAUTHORdetails{%
  title = {Fuglede's conjecture on cyclic groups of order $p^n q$}, 
  author = {Romanos Diogenes Malikiosis and Mihail N. Kolountzakis},
  plaintextauthor = {Romanos Diogenes Malikiosis and Mihail N. Kolountzakis},
  plaintexttitle = {Fuglede's conjecture on cyclic groups of order p**n q}, 
  keywords = {Spectral set, tiling in finite abelian groups, Fuglede's conjecture, vanishing sums of roots of unity},
}   

\dajEDITORdetails{%
   year={2017},
   number={12},
   received={7 December 2016},   
   published={5 September 2017},  
   doi={10.19086/da.2071},       
}   

\begin{document}

\begin{frontmatter}[classification=text]


\author[RDM]{Romanos Diogenes Malikiosis\thanks{Supported by a Postdoctoral Fellowship from Alexander von Humboldt Foundation.}}
\author[MNK]{Mihail N. Kolountzakis\thanks{Received partial support for this research by grant No 4725 of the University of Crete and Alexander von Humboldt Foundation.}}

\begin{abstract}
We show that
the spectral set conjecture by Fuglede \cite{Fuglede} holds in the setting of
cyclic groups of order $p^nq$, where $p$, $q$ are distinct primes and $n\geq1$.
This means that a subset $E$ of such a group $G$ tiles the group by translation
($G$ can be partitioned into translates of $E$) if and only if
there exists an orthogonal basis of $L^2(E)$ consisting of group characters.
The main ingredient of the present proof is the structure of vanishing sums of
roots of unity of order $N$, where $N$ has at most two prime divisors; the
extension of this proof to the case of cyclic groups of order $p^nq^m$ seems
therefore feasible. The only previously known infinite family of cyclic groups,
for which Fuglede's conjecture is verified in both directions, is that of cyclic $p$-groups, i.e.
$\ZZ_{p^n}$.
\end{abstract}
\end{frontmatter}


\section{Introduction}

Let $\Om$ be a 
measurable subset of $\RR^n$ of positive Lebesgue measure.
$\Om$ is called \emph{spectral}, if it accepts an orthogonal basis of
exponentials, namely $e^{i\la\cdot x}$, where $\la$ ranges through
$\La\ssq\RR^n$; the set $\La$ is called the \emph{spectrum} of $\Om$. On the
other hand, $\Om$ is called a \emph{tile} of $\RR^n$, if there is a set
$T\ssq\RR^n$ such that almost every point of $\RR^n$ can be written uniquely as
$\om+t$, where $\om\in \Om$ and $t\in T$; $T$ is hence called the \emph{tiling
complement} of $\Om$.

Fuglede's spectral set conjecture \cite{Fuglede} asserts that $\Om$ is spectral
if and only if $\Om$ is a tile. For many years, there was a lot of positive
evidence towards the veracity of this conjecture; Fuglede himself proved this
conjecture when the spectrum or the tiling complement of $\Om$ is a lattice.
Later on, the conjecture was proved for various families of subsets by numerous
authors, for example, for $2$-dimensional convex bodies \cite{IKTconvexplanar},
unions of two intervals in $\RR$ \cite{Laba2intervals}, etc.

Against the current of research on this subject, Tao disproved this conjecture
for dimensions $5$ and above \cite{Tao04}, constructing a spectral set in $5$
dimensions that does not tile the space\footnote{Any counterexample in $n$
dimensions could be extrapolated to a counterexample in all higher
dimensions.}. Subsequently, works by Kolountzakis and Matolcsi
\cite{KMhadamard,KM06,Mat} and Farkas-Matolcsi-Mora \cite{FMM06} showed that
the conjecture is false in dimensions $3$ and above in both directions, leaving
it open for $\RR$ and $\RR^2$.

This conjecture can be naturally stated for other groups, for example $\ZZ$ or
any 
locally compact abelian group. These cases are not only
interesting on their own, but they also have connections to the original case. In
his disproof of the $5$-dimensional case, Tao constructed a spectral set in
$\ZZ_3^5$ containing $6$ elements, hence not a tile, as the cardinality of any
tile of a finite abelian group divides the order of the group; then he lifted
this counterexample to $\RR^5$. So was done in the disproof of the other direction.

Some examples of note, where Fuglede's conjecture holds, include finite cyclic
$p$-groups \cite{Laba,Qpfuglede}, $\ZZ_p\times\ZZ_p$ \cite{IMP15}, and $\QQ_p$
\cite{Qpfuglede2,Qpfuglede}, the field of $p$-adic numbers. Borrowing the
notation from \cite{DL14}, we write {\bf S-T}$(G)$, respectively {\bf
T-S}$(G)$, if the Spectral$\Rightarrow$Tile direction, respectively
Tile$\Rightarrow$Spectral, holds in $G$; when we put $G=\ZZ_N$, it is
understood that the statement holds \emph{for all} $N$.  The connection between
the conjecture on $\RR$ and on finite cyclic groups or $\ZZ$ is summarized
below \cite{DL14}:
\[\text{{\bf T-S}}(\RR)\Longleftrightarrow\text{{\bf T-S}}(\ZZ)\Longleftrightarrow\text{{\bf T-S}}(\ZZ_N),\]
and
\[\text{{\bf S-T}}(\RR)\Longrightarrow\text{{\bf S-T}}(\ZZ)\Longrightarrow\text{{\bf S-T}}(\ZZ_N).\]

According to the above connections, a counterexample in a finite cyclic group can be lifted to a counterexample in $\RR$; on the other hand, if the conjecture were true for every
cyclic group and $\ZZ$, this would hold no meaning for the original conjecture in $\RR$, unless it were proven that every spectral set in $\RR$ has a rational spectrum. 
We may ask nevertheless, to which extent is Fuglede's conjecture true for finite abelian
groups, or even cyclic ones. Surprisingly, not much is known for cyclic groups, apart from cyclic $p$-groups, i.e. $\ZZ_{p^n}$
for $p$ prime. The direction Tile$\Rightarrow$Spectral is known also for cyclic groups of order $p^nq^m$, 
for $p$, $q$ distinct primes \cite{CM,Laba}; see Section \ref{TS} below. 
\begin{thm}[\cite{CM,Laba}]\label{knownthm}
 Let $A\ssq\ZZ_N$ be a set that tiles $\ZZ_N$ by translations, where $N=p^n q^m$ for distinct primes $p$, $q$. Then $A$ is spectral.
\end{thm}

The novel contribution of our present work is the proof of Spectral$\Rightarrow$Tile direction for cyclic groups of order $N=p^nq$, thus establishing the veracity of Fuglede's
conjecture in this setting. 
\begin{thm}\label{mainthm}
 Let $A\ssq\ZZ_N$ be a spectral set, where $N=\pnq$ for distinct primes $p$, $q$. Then $A$ tiles $\ZZ_N$ by translations.
\end{thm}
The proof relies heavily on the structure of vanishing sums of roots of unity of order $N$, where $N$ is divided by at most two primes. The fact that
such sums are \emph{nonnegative} linear combinations of $p$- and $q$-cycles \cite{LL} gives efficient bounds on spectral sets $A\ssq\ZZ_N$; we believe that these techniques can be extended
to every cyclic group of order $N=p^nq^m$. Unfortunately, we have not managed to conclude our proof in this more general setting so far.

The proof for the direction Tile$\Rightarrow$Spectral for $N=p^nq^m$ seems easier for the following reasons: firstly, it is facilitated greatly by the conditions {\bf (T1)} and {\bf (T2)} to be
defined in Section \ref{prelim}, that are equivalent to the tiling condition \cite{CM} when $N$ has at most two distinct prime factors and thus impose a very strong structure on the 
tiling set $A$, yielding also the spectrality thereof (not to mention that the tiling condition alone yields $\#A\mid N$, restricting the size of $A$ down to the divisors of $N$); 
secondly, although spectrality gives some structure on $A$, connecting the set of roots of the {\em mask polynomial} of $A$ with the difference set of a spectrum $B$, it is difficult to
obtain new information about $A$ when $A(X)$ vanishes on primitive root of unity whose order is not prime power. We overcome this obstacle by virtue of the structure of the vanishing sums
of roots of unity, obtaining strong lower bounds for the size of the spectrum, leading to a contradiction whenever one of the conditions {\bf (T1)} and {\bf (T2)} is not met. The latter method
is most efficient when 
\[\wh\chi_A(d_1)=\wh\chi_A(d_2)=\dotsb=\wh\chi_A(d_k)=0,\]
for $d_1\mid d_2\mid \dotsb \mid d_k \mid N$ and $k$ relatively large. This successive divisibility property occurs much more frequently when the power of $p$ or $q$ in the factorization of $N$
is small, and this limitation did not allow us to go further than the case $N=p^n q$, for the time being.

Lastly, we have to emphasize that these techniques can be no further extended. Consider the following vanishing sum of roots of unity of order $N=pqr$, where $p$, $q$, $r$
are distinct primes:
\[(\om_p+\om_p^2+\dotsb+\om_p^{p-1})(\om_q+\om_q^2+\dotsb+\om_q^{q-1})+(\om_r+\om_r^2+\dotsb+\om_r^{r-1})=(-1)(-1)+(-1)=0,\]
where $\om_n=e^{2\pi i/n}$.
It is known that this sum cannot be expressed as a nonnegative linear combination of $p$-, $q$-, or $r$-cycles \cite{LL}, therefore we cannot obtain strong lower bounds on the size of a
spectral subset of $\ZZ_N$. Whether sums such as the above lead to a counterexample in some finite cyclic group and eventually in $\RR$, remains to be seen.

\bigskip
\bigskip

\section{Preliminaries}\label{prelim}
\bigskip
\bigskip
 
 Let $\ZZ_N$ denote the ring of integers modulo $N$. With every (multi)set $A$ with elements from $\ZZ_N$, we associate a polynomial in the quotient ring $R=\ZZ[X]/(X^N-1)$, say
 \[A(X)=\sum_{a\in A}m_aX^a,\]
 where $m_a$ is the multiplicity of $a$ in the multi-set $A$. $A$ is a proper set, if and only if $A(X)$ has coefficients $0$ and $1$ (it is understood that we write any element
 in $R$ as a linear combination of $1,X,\dotsc,X^{N-1}$). $A(X)$ is called the \emph{mask polynomial} of $A$; it has the following connection with the Fourier transform of
 the characteristic function of $A$:
 \[\wh\chi_A(n)=A(\om^{-n}),\]
 where $\om=e^{2\pi i/N}$ throughout this paper.
 
 A subset $A\ssq\ZZ_N$ is called \emph{spectral} if there is
 a set $B$ with $\#A=\#B$, such that the set of exponentials 
 \[x\mapsto e^{\frac{2\pi ibx}{N}},\	b\in B\]
 is orthogonal on $A$ with the usual inner product.
 \begin{thm}\label{diff}
  Let $A\ssq\ZZ_N$ be spectral. Then
 \[B-B\ssq\set{0}\cup\set{n:A(\om^{n})=0}.\] 
 \end{thm}

 \begin{proof} 
By definition, we get
 \[\sum_{a\in A}e^{\frac{2\pi i(b-b')a}{N}}=0,\]
 whenever $b,b'$ are distinct elements of $B$. This is equivalent to the condition
 \[B-B\ssq\set{0}\cup\set{n:A(\om^{n})=0}.\qedhere\]
 \end{proof}

 There is a natural action of the Galois group
 \[G=\Gal(\QQ(\om)/\QQ)\cong\ZZ_N^{\star}\]
 on the values of $A(X)$, given by
 \begin{equation}\label{galaction}
 \sig(A(\om^a))=A(\om^{ag}),
 \end{equation}
 where $\sig\in G$ is determined by $\sig(\om)=\om^g$, for some $g\in\ZZ_N^{\star}$. Therefore, in order to determine the
 support of $A(\om^n)$ we only need to evaluate at the divisors of $N$.
 For any integer $N$, the \emph{divisor class} of $N$ with respect to $n>0$, a divisor of $N$, is simply $n\ZZ_N^{\star}$, which is the set of residues $\bmod N$
 whose greatest common divisor with $N$ is equal to $n$. The set $\set{n:A(\om^n)=0}$ is just a union of divisor classes. This is equivalent to the fact that
 if $A(\om^d)=0$ for some $d\mid N$, then $\Phi_{N/d}(X)\mid A(X)$, where $\Phi_n(X)$ denotes the $n$-th cyclotomic polynomial.
 
 We also denote 
 \[D:=\set{n\in\NN : n\mid N, A(\om^m)\neq0, \text{ for all }m\text{ with }n\mid m \text{ and } m\mid N}.\]
  \begin{prop}\label{D0}
  If $A\ssq\ZZ_N$ is spectral and $n\in D$, then $\#A\leq n$.
 \end{prop}
 \begin{proof}
  By Theorem \ref{diff} and the hypothesis we get
  \[(B-B)\cap n\ZZ_N=\set{0}\]
  for a spectrum $B$ of $A$. Hence, no two distinct elements of $B$ can have the same residue $\bmod n$, thus obtaining $\#A=\#B\leq n$.
 \end{proof}
 
 
 The following properties on polynomials were introduced by Coven and Meyerowitz \cite{CM} in their effort to characterize finite sets that tile the integers by
 translations. We will adapt this definition for subsets of $\ZZ_N$.
 
 \begin{defn}
  Let $A\ssq\ZZ_N$, and let $A(X)$ be its mask polynomial. The set of prime powers dividing $N$ is denoted by $S$; define
  \[S_A=\set{s\in S:\Phi_s(X)\mid A(X)}.\]
  We say that $A$ satisfies the property {\bf (T1)}, if 
  \[\#A=A(1)=\prod_{s\in S_A}\Phi_s(1),\]
  and that it satisfies {\bf (T2)}, if for every distinct elements $s_1,\dotsc,s_m$ of $S_A$, $\Phi_{s_1\dotsm s_m}(x)$ divides $A(x)$.
 \end{defn}

 \begin{thm}[Theorem A \cite{CM}]\label{t1t2}
  If $A\ssq\ZZ_N$ satisfies {\bf (T1)} and {\bf (T2)}, then $A$ tiles $\ZZ_N$ by translations.
 \end{thm}

 Let $d\mid N$. A $d$-cycle is a coset by the cyclic subgroup of $N$ with $d$ elements, that is, a set of the form
 \[\set{j,j+N/d,j+2N/d,\dotsc,j+(d-1)N/d}.\]
 Especially in the case when $N$ has only two prime divisors, the following Lemma allows us to discern the structure of $n\cdot A:=\set{na:a\in A}$ (a multi-set), 
 whenever $A(\om^n)=0$. In particular, it says that $n\cdot A$ must be the union of $p$- and $q$-cycles.

 \begin{lemma}\label{pqcycles}
 Let $n\mid N$ be such that $N/n$ has at most two prime divisors, say $p$ and $q$. If $A(\om^n)=0$, then 
 \begin{equation}\label{cycles}
 A(X^n)\equiv P_n(X^n)\Phi_p(X^{N/p})+Q_n(X^n)\Phi_q(X^{N/q})\bmod X^N-1,
 \end{equation}
 where $P_n$ and $Q_n$ have nonnegative coefficients.
 \end{lemma}
 \begin{proof}
  By definition, $A(\om^n)$ is a vanishing sum of roots of unity of order $N/n$, in particular
  \[0=A(\om^n)=\sum_{a\in A}\om^{na}.\]
  As $N/n$ has at most two prime divisors, $p$ and $q$, the above sum can be written as linear combination with \emph{nonnegative integer coefficients} of sums of the form
  \[\om^k(1+\om_p+\om_p^2+\dotsb+\om_p^{p-1})\]
  or
  \[\om^k(1+\om_q+\om_q^2+\dotsb+\om_q^{q-1}),\]
  due to Theorem 3.3 from \cite{LL}, which shows that $n\cdot A$ is the union of $p$- and $q$-cycles, as multi-sets. Every $p$-cycle
  has a mask polynomial of the form $X^k\Phi_p(X^{N/p})$; union of multi-sets corresponds to addition of the mask polynomials, thus obtaining \eqref{cycles}. We note that
  the argument of $P_n,Q_n$ is $X^n$, simply because $A(X^n)$ can be expressed in terms of powers of $X^n$, as $n\mid N$.
 \end{proof}

 \begin{rem}
  If $N/n$ has only one prime divisor, say $p$, then it is understood that $Q_n\equiv 0$.
 \end{rem}
 
 It is also useful to find conditions under which $n\cdot A$ cannot be written as a union of $p$-cycles or $q$-cycles only, or equivalently, $P_nQ_n\not\equiv 0$, for every such
 possible choice of $P_n$ and $Q_n$.
 
 \begin{prop}\label{PQ}
  Let $N$ have only two prime divisors, say $p$ and $q$, and $A(\om^n)=0$, for some $n\mid N$, so that
  \[A(X^n)\equiv P_n(X^n)\Phi_p(X^{N/p})+Q_n(X^n)\Phi_q(X^{N/q})\bmod X^N-1.\]
  If $A(\om^{np^a})\neq0$ for some $a>0$, then $P_n\not\equiv0$, and if $A(\om^{nq^b})\neq0$ for some $b>0$, then $Q_n\not\equiv0$.
 \end{prop}
 \begin{proof}
  If $Q_n\equiv0$, then 
  \[A(X^{nq^b})\equiv P_n(X^{nq^b})\Phi_p(X^{Nq^b/p})\equiv P_n(X^{nq^b})\Phi_p(X^{N/p}),\]
  thus obtaining $A(\om^{nq^b})=0$. The other case is proven similarly.
 \end{proof}

 {\bf (T1)} and {\bf (T2)} are conjectured to hold if and only if $A$ tiles $\ZZ_N$ \cite{CM} (this conjecture was initially formulated in $\ZZ$). For every set $A\ssq\ZZ_N$,
 however, a weaker property than {\bf (T1)} holds, that is very useful for bounding $\#A$.
 
 \begin{prop}\label{lower}
  Let $A\ssq\ZZ_N$. Then $\prod_{s\in S_A}\Phi_s(1)$ divides $\#A$. In particular, if $p$ is a prime divisor of $N$, and there are $m$ integers $0<a_1<a_2<\dotsb<a_m$
  such that $A(\om^{N/p^{a_i}})=0$ for all $1\leq i\leq m$, then $p^m\mid \#A$, where $p^{a_m}\mid N$.
 \end{prop}
 \begin{proof}
  By definition, $A(X)$ is divided by $\prod_{s\in S_A}\Phi_s(X)$ in $\ZZ[X]$. Putting $X=1$, we get the desired result.
 \end{proof}

 In the rest of this paper we will prove Fuglede's conjecture on $\ZZ_N$, where $N=p^nq$, for $p\neq q$ primes. The direction Tile$\Rightarrow$Spectral can be deduced from
 the work of Coven-Meyerowitz \cite{CM} and \L{}aba \cite{Laba} in the more general case when $N$ has at most two prime divisors, and is shown in Section \ref{TS}.
 The direction Spectral$\Rightarrow$Tile is proven in section \ref{ST}. For completeness, we will also present the Spectral$\Rightarrow$Tile proof for $N=p^n$ in
 section \ref{ST0} due to its elegance and brevity, although a proof for this case follows from \cite{Laba}; a different proof also appeared in \cite{Qpfuglede}.
 
 One final tool that will be very useful in this note, is the following:
 
 \begin{lemma}\label{combo}
  Let $m,n>0$ be two relatively prime integers, and $0<k<mn$ another integer. Then, there is at most one pair $(s,t)$ of nonnegative integers, such that $k=sm+tn$. If $k=mn$,
  then there are exactly two such pairs, namely $(n,0)$ and $(0,m)$.
 \end{lemma}
 
 \begin{proof}
  Assume that $0<k<mn$ and there is a pair $(s,t)$ such that $k=sm+tn$, with $s,t\geq 0$. All other pairs of integer solutions $(s',t')$ to the Diophantine equation
  $k=s'm+t'n$, satisfy $s'=s-nx$, $t'=t+mx$, for $x\in\ZZ$. If $x>0$, then $s'\leq s-n\leq k/m-n<0$, whereas if $x<0$ we get $t'\leq t-m\leq k/n-m<0$. The case $k=mn$
  is proven similarly.
 \end{proof}

\bigskip
\bigskip

 \section{Tile$\Rightarrow$Spectral, $N=p^nq^m$}\label{TS}
 \bigskip
 \bigskip

 In this section, we will review the proof of Theorem \ref{knownthm}.
 The proof is not new, and is based on combined arguments from \cite{CM} and \cite{Laba}, even though the case for finite cyclic groups is not explicitly mentioned. 
 We will need the following tools from \cite{CM} and \cite{Laba}.
 
 \begin{lemma}[Lemma 1.3, \cite{CM}]\label{13}
   Let $N$ be an integer and let $A$ and $B$ be finite multisets of nonnegative integers with mask polynomials $A(X)$ and $B(X)$. Then the
following statements are equivalent. Each forces $A$ and $B$ to be sets such that $\#A\#B=A(1)B(1)=N$.
\begin{enumerate}
 \item $A\oplus(B\oplus N\ZZ)=\ZZ$ is a tiling.
 \item $A\oplus B$ is a complete set of residues modulo $N$.
 \item $A(X)B(X)\equiv1+X+\dotsb X^{N-1}\bmod X^N-1$.
 \item $N=A(1)B(1)$ and for every factor $t>1$ of $N$, the cyclotomic polynomial $\Phi_t(X)$ is a divisor of $A(X)$ or $B(X)$.
\end{enumerate}
\end{lemma}

\begin{lemma}[Lemma 2.1, \cite{CM}]\label{21}
 Let $A(X)$ and $B(X)$ be polynomials with coefficients 0 and 1, $N=A(1)B(1)$, and $S$ the set of prime power factors of $N$. If $\Phi_t(X)$ divides
$A(X)$ or $B(X)$ for every factor $t>1$ of $N$, then
\begin{enumerate}
 \item $A(1)=\prod_{s\in S_A}\Phi_s(1)$ and $B(1)=\prod_{s\in S_B}\Phi_s(1)$.
 \item $S_A$ and $S_B$ are disjoint sets whose union is $S$.
\end{enumerate}
\end{lemma}

Now, assume that $A$ tiles $\ZZ_N$ by translations, and let $B$ be the complementary tile.
We may assume that $A\ssq\set{0,1,2,\dotsc,N-1}$, and also assume that $A$ tiles $\ZZ$ by translations; furthermore, this tiling has period $N$,
i.~e. $A\oplus(B\oplus N\ZZ)=\ZZ$. We warn the reader, that not only
do we have to prove that $A$, as a subset of $\ZZ$, is spectral, but also that the spectrum is a subset of $N^{-1}\ZZ$, in order to claim that $A$,
\emph{as a subset of $\ZZ_N$}, is spectral.

We see that conditions $(1)$ and $(2)$ of Lemma \ref{13} are satisfied, hence $(4)$ is satisfied as well, which is just the hypothesis of Lemma \ref{21}.
By $(2)$ of Lemma \ref{21}, we get that $S_A\ssq S$. Next, we will use the following two theorems:

\begin{thm}[Theorems B1 and B2, \cite{CM}]\label{CoMe}
 Let $A$ be a finite set of nonnegative integers with corresponding polynomial $A(x)=\sum_{a\in A}x^a$. If $A$ tiles the integers, and $\#{A}$ is divided
 by at most two primes, then $A$ satisfies {\bf (T1)} and {\bf (T2)}.
\end{thm}

\begin{thm}[Theorem 1.5(i), \cite{Laba}]\label{La}
 If $A\ssq\ZZ$ satisfies {\bf (T1)} and {\bf (T2)}, then $A$ has a spectrum.
\end{thm}

\begin{rem}
The important part of the proof, is that the spectrum is explicitly constructed with respect to $S_A$, namely, the set of all
\[\sum_{s\in S_A}\frac{k_s}{s},\]
where $k_s\in\set{0,1,\dotsc,p-1}$, $s\in S_A$ and $s=p^{\al}$, is proven to be a spectrum of $A$, when it satisfies {\bf (T1)} and {\bf (T2)} (see the beginning of Section 2 \cite{Laba}).
\end{rem}

So, if $A\ssq\set{0,1,\dotsc,N-1}$ tiles $\ZZ$ by translations, the tiling having period $N$, which has at most two prime divisors, then $A$ satisfies (T1)
and (T2) by Theorem \ref{CoMe}, as $\#{A}$ divides $N$. Next, by Theorem \ref{La} we get that $A$ is spectral; by Lemma \ref{21} and the Remark above, we get
that the spectrum belongs to $N^{-1}\ZZ$, hence $A\ssq\ZZ_N$ is spectral, completing the proof.

\bigskip
\bigskip
\section{Spectral$\Rightarrow$Tile, $N=p^n$}\label{ST0}

\bigskip
\bigskip

Suppose $\Lambda \subseteq \ZZ_N$ is a spectrum of $A$ and let $p^{\nu_1}, \ldots, p^{\nu_k}$
be the divisors $d$ of $N$ such that
\[
\Phi_d(X) \mid A(X) = \sum_{a \in A} X^a.
\]
We have
\[
\Phi_{p^{\nu_i}}(X) = 1+X^{p^{\nu_i-1}}+X^{2p^{\nu_i-1}}+\cdots+X^{(p-1)p^{\nu_i-1}}.
\]
Write $E_\nu = \{0, p^{\nu-1}, 2p^{\nu-1}, \ldots, (p-1)p^{\nu-1}\}$ so that
$E_{\nu_i}(X) = \Phi_{p^{\nu_i}}(X)$. Define next the set
\begin{equation}\label{sum-of-e}
E = E_{\nu_1} + E_{\nu_2} + \cdots E_{\nu_k},
\end{equation}
and notice the sum is direct as $e_1+\cdots+e_k \in E$ is determines
$e_i \in E_{\nu_i}$ from the ${\nu_i}$-th digit in its expansion to base $p$.
This observation implies that $|E| = p^k$.

Notice also that $A(X)$ and $E(X)$ have the same zeros at the $N$-th roots of unity.
This implies that $\Lambda$ is also orthogonal on $E$ as this is determined by the zeros
of $E(X)$ at the $N$-th roots of unity.
From the orthogonality we obtain
\[
|A| = |\Lambda| \le |E| = p^k.
\]
Let $B \subseteq \ZZ_N$ be the sum of those $E_\nu$, $\nu=1,2,\ldots,n$, not appearing in \eqref{sum-of-e}.
This sum is again direct, as with the sum \eqref{sum-of-e}, so we obtain
\[
|B| = p^{n-k}.
\]
It follows that $A(X)B(X)$ vanishes on all $N$-th roots of unity except 1, which implies that
$A+B$ is a tiling of $\ZZ_N$ at some level $\ell$.
Then
\[
\ell p^n = |A| \cdot |B| = p^k p^{n-k} = p^n,
\]
so that $\ell=1$ and $A+B$ is a tiling of $\ZZ_N$ at level 1.

\bigskip
\bigskip
\section{Spectral$\Rightarrow$Tile, $N=p^nq$}\label{ST}
\bigskip
\bigskip

In this Section we prove Theorem \ref{mainthm}. As mentioned in the Introduction, our method is to establish a contradiction to the assumption of a spectral set not satisfying {\bf (T1)}
and {\bf (T2)}. When one of these properties fails, the spectrality of $A$ induces the existence of many roots of the mask polynomial $A(X)$ of the form
\begin{equation}\label{succroots}
A(\om^{d_1})=\dotsb=A(\om^{d_k})=0,
\end{equation}
where $d_1\mid\dotsb\mid d_k\mid N$, which further gives
\begin{equation}\label{cyclodiv}
\Phi_{N/d_1}(1)\dotsm\Phi_{N/d_k}(1)\mid \#A 
\end{equation}
and this produces a strong lower bound for $\#A$ via the vanishing sums of roots of unity which eventually leads to a contradiction. Obtaining such lower bounds were
a lot easier for $N=p^n$, as \eqref{cyclodiv} would yield $p^k\mid\#A$. This follows from the fact that $\Phi_{p^s}(1)=p$, for a prime $p$, whereas $\Phi_M(1)=1$, when $M$ is not a prime power, and this 
explains the added complexity of this proof, especially when the fractions $N/d_i$ are not prime powers. Lemma \ref{pqcycles} is the main tool here; although it does not produce some divisibility 
condition such as \eqref{cyclodiv}, it gives some good lower bounds in terms of $k$ when \eqref{succroots} holds. See for example, the proofs of the Claims \ref{firstbound} and \ref{secondbound} below.

We distinguish two cases, depending on whether $q$ belongs to $D$.

\bigskip\noindent
\framebox{$q\in D$}
We have $\#A\leq q$ from Proposition \ref{D0}. Furthermore, the property {\bf (T2)} holds vacuously, so we only need to prove {\bf (T1)} due to
Theorem \ref{t1t2}. If $\#A=1$, $A(X)$ is a monomial and has no root of the form $\om^d$, in particular, $A(\om^{p^n})\neq0$ and {\bf (T1)} holds. 

If $\#A>1$, and $A$ is spectral, then $B-B\neq\set{0}$ for a spectrum $B$, so $A(\om^d)$ must vanish somewhere. Since
$q\in D$, there must be some nonnegative $a\leq n$ such that $A(\om^{p^a})=0$, so that
\[A(X^{p^a})\equiv P(X^{p^a})\Phi_p(X^{N/p})+Q(X^{p^a})\Phi_q(X^{N/q})\mod X^N-1,\]
by Lemma \ref{pqcycles}. $q\in D$ yields $A(\om^{p^aq})\neq0$, hence
by Proposition \ref{PQ} we get $Q\not\equiv0$, so that $\#A=A(1)=pP(1)+qQ(1)\geq q$, leading to $P\equiv0$ and $Q(1)=1$. 
This certainly implies that $A(\om^{p^n})=0$, $A(1)=q$, hence {\bf (T1)} holds in any case where $A$ is spectral and $q\in D$. By Theorem \ref{t1t2}, $A$ tiles $\ZZ_N$ by translations.


\bigskip\noindent
\framebox{$q\notin D$}
The size of a spectral set $A\ssq\ZZ_N$ depends on the number of roots of $A(X)$ of the form $\om^{p^aq}$.

\begin{defn}
 The numbers $0\leq a_1<\dotsb<a_m\leq n-1$ are such that $\set{a: A(\om^{p^aq})=0}=\set{a_1,\dotsc,a_m}$. If no such root exists, we simply put $m=0$.
 Also, if $A$ is spectral and $B$ a spectrum of $A$, we denote $B_i=\set{b\in B:b\equiv i\bmod q}$. The $p$-adic expansion of
 the least nonnegative residue $\bmod p^n$ of $b\in B$ has the form
 \[b\equiv b_0+b_1p+\dotsb+b_{n-1}p^{n-1}\bmod p^n,\]
 where $0\leq b_i\leq p-1$ for all $i$. We say that the $p$-adic expansions of $b$ and $b'$ coincide at $a_1,\dotsc,a_m$, if $b_{a_i}=b'_{a_i}$ for $1\leq i\leq m$, that is,
 they have the same $p$-adic digits at those places. Finally, $p^i\pdiv b$ exactly when $b_i$ is the smallest nonzero $p$-adic digit of $b$.
\end{defn}

We have seen in the previous subsection that Spectral$\Rightarrow$Tile holds when $m=0$. By induction, we may assume that Spectral$\Rightarrow$Tile holds for all
nonnegative integers up to $m-1$, for some $m>0$, where $\set{a: A(\om^{p^aq})=0}=\set{a_1,\dotsc,a_m}$, exactly as in the Definition above.

\begin{prop}\label{main}
 Suppose that $A\ssq\ZZ_N$ is spectral and $\#A>p^m$. Then
 \begin{enumerate}
  \item $\#A=p^mq$.
  \item $\#B_i=p^m$ for all $i$, and the elements of $B_i$ have all the possible $p$-adic expansions at $a_1,\dotsc,a_m$, each appearing exactly once.
  \item For every $i,j$ and every $b\in B$, there is $b'$ such that $p^{a_j}\pdiv b-b'$.
  \item There is some $a\neq a_i$ for all $i$, such that $b-b'\in p^a\ZZ_N^{\star}$ for some $b,b'\in B$, and $A(\om^{p^a})=0$.
  \item If $a>a_m$, then $A(\om^{p^{a_j}})=0$ for all $j$, as well as $A(\om^{p^{a'}})=0$ for all $a'\geq a$; in particular, $A(\om^{p^n})=0$, or equivalently, $p^n\notin D$.
  \item If $a<a_m$, then $A(X^{p^{a_m}q})\equiv p^{m-1}qX^k\Php \bmod X^N-1$, for some $k$.
 \end{enumerate}
\end{prop}

\begin{proof}
 As $\#B>p^m$, there are at least two
distinct elements of $B$, say $b,b'$, that have the same digits at places $a_1,a_2,\dotsc,a_m$. If $q\mid b-b'$, then $b-b'\notin \bigcup_{1\leq i\leq m}p^{a_i}q\ZZ_N^{\star}$, 
contradicting the fact that $A(X)$ has exactly $m$ roots of the form $\om^{p^{a}q}$, due to Theorem \ref{diff}. Hence $q\nmid b-b'$, and $b-b'\in p^a\ZZ_N^{\star}$,
so that $A(\om^{p^a})=0$ by Theorem \ref{diff}, where $a$ is the
first place where the $p$-adic expansions of $b$ and $b'$ differ, therefore $a\neq a_i$ for all $i$, proving (4).

With the same argument we can show that $\#B_i\leq p^m$ for all $i$; in this case any two elements $b,b'\in B_i$ satisfy $q\mid b-b'$, so by Theorem \ref{diff} and the hypothesis,
they must have at least one different digit at places $a_1,\dotsc,a_m$, yielding $\#B_i\leq p^m$, and $\#B=\#A\leq p^mq$. For convenience, put $n_i=p^{a_i}q$ and $d=p^a$.
By Lemma \ref{pqcycles} we get
\begin{equation}\label{axd}
 A(X^d)\equiv P_d(X^d)\Phi_p(X^{N/p})+Q_d(X^d)\Phi_q(X^{N/q})\bmod X^N-1,
 \end{equation}
where $Q_d\not\equiv0$ due to Proposition \ref{PQ}, as $A(\om^{p^aq})\neq0$.

Now, consider the largest index $i$, such that $a>a_i$, assuming first that there is such an index, i.~e. $a>a_1$; otherwise, we put $i=0$. Put $u=dq=p^aq$, and denote by $\norm{A(X)}_{\infty}$
the largest coefficient of $A(X)$ in $R=\ZZ[X]/(X^N-1)$ written as a linear combination of $1,X,\dotsc,X^{N-1}$.

\begin{claim}\label{firstbound}
 $\norm{A(X^u)}_{\infty}\geq p^iq$.
\end{claim}
\begin{proof}[Proof of Claim]
 From \eqref{axd} we obtain 
 \begin{equation}\label{axn1}
  A(X^u)\equiv P_d(X^u)\Phi_p(X^{N/p})+qQ_d(X^u)\bmod X^N-1.
  \end{equation}
 If $i=0$, then from $Q_d\not\equiv0$ we get that some coefficient of $A(X^u)$ is at least as large as $q$, as desired. Suppose that $i>0$. By repeated application of
 Lemma \ref{pqcycles}, we have
 \[A(X^{n_j})\equiv P_{n_j}(X^{n_j})\Php,\]
 for all $j$. For $j=1$, if we replace $X$ by $X^{n_2/n_1}$, we also get $A(X^{n_2})\equiv pP_{n_1}(X^{n_2})$, and comparing this with $A(X^{n_2})\equiv P_{n_2}(X^{n_2})\Php$,
 we deduce that $\Php$ divides $P_{n_1}(X^{n_2})$, thus obtaining 
 \[A(X^{n_2})\equiv pP_{n_1}^1(X^{n_2})\Php,\]
 for some polynomial $P_{n_1}^1$ with positive integer
 coefficients. Proceeding inductively, we can get
 \[A(X^{n_i})\equiv p^{i-1}P_{n_1}^{i-1}(X^{n_i})\Php,\]
 hence
 \begin{equation}\label{axn2}
 A(X^u)\equiv p^iP_{n_1}^{i-1}(X^u),
 \end{equation}
 since $a_i<a<a_{i+1}$, where $P_{n_1}^{i-1}$ has positive integer coefficients. Comparing \eqref{axn1} and \eqref{axn2}, we get that $p^i$ divides all coefficients of $A(X^u)$.
 Furthermore, since $A(\om^u)\neq0$, we deduce that there is at least a $p$-cycle on which the elements of $n_m\cdot A$ do not have the same multiplicity; using \eqref{axn1} again,
 we deduce that two such multiplicities must differ by a multiple of $q$. Therefore, by \eqref{axn2} we conclude that their difference is a multiple of $p^iq$, and since they are both
 nonnegative and distinct, we finally get $\norm{A(X^u)}_{\infty}\geq p^iq$.
\end{proof}

If $a>a_m$ the claim yields $\normoo{A(X^u)}\geq p^mq$, while on the other hand $A(1)\leq p^mq$, therefore, $A(X^u)\equiv p^mq X^{uk}$, for some $k$. Then, \eqref{axn1}
yields $P_d\equiv0$, so for all $a'\geq a$ \eqref{axn1} gives
\[A(X^{p^{a'}})\equiv Q_d(X^{p^{a'}})\Phi_q(X^{N/q})\bmod X^N-1\]
establishing (5) (and (1) when $a>a_m$).

For the rest of the proof, we suppose $a<a_m$, and we will estimate $\norm{A(X^{n_m})}_{\infty}$.

\begin{claim}\label{secondbound}
 $\normoo{A(X^{n_{j+1}})}\geq p\normoo{A(X^{n_j})}$.
\end{claim}
\begin{proof}[Proof of Claim]
 Since $A(\om^{n_j})=0$, we obtain from Lemma \ref{pqcycles} $A(X^{n_j})\equiv P_{n_j}(X^{n_j})\Php$. The largest coefficient of $A(X^{n_j})$ would appear in a $p$-cycle, hence
 $\normoo{A(X^{n_j+1})}\geq p\normoo{A(X^{n_j})}$, and finally, $\normoo{A(X^{n_{j+1}})}\geq\normoo{A(X^{n_j+1})}\geq p\normoo{A(X^{n_j})}$, as desired.
\end{proof}

Applying Claims \ref{firstbound} and \ref{secondbound}, we obtain $\normoo{A(X^{n_m})}\geq p^{m-1}q$. The largest coefficient of $A(X^{n_m})$ would appear on a $p$-cycle, since $A(\om^{n_m})=0$, so we would
get $\#A=A(1)\geq p^mq$; but we have already shown that $\#A\leq p^mq$, so $A(1)=p^mq$ and $A(X^{n_m})\equiv p^{m-1}qX^k\Php$ thus proving (6) and (1) in all cases. 
(2) follows immediately from (1), as $\#B=p^mq$, hence $\#B_i=p^m$ for 
all $i$. Finally, (3) is a direct consequence of (2); let $i$ and $b\in B_i$ be arbitrary. For every $j$, there is some $b'\in B_i$ whose $p$-adic expansion is the same on
$a_1,\dotsc,a_{j-1}$ but differs on $a_j$, due to (2). Then, $b-b'\in p^cq\ZZ_N^{\star}$, for some $c\leq a_j$. By Theorem \ref{diff} and the hypothesis, $c=a_l$, for some $l$, but 
$c=a_l$ cannot hold for $l<j$, thus $l=j$, completing the proof.
\end{proof}

Next, we will make use of the induction assumption.

\begin{prop}\label{partition}
 Let $A\ssq\ZZ_N$ spectral and $\set{a: A(\om^{p^aq})=0}=\set{a_1,\dotsc,a_m}$. Define a partition of $A$ into sets $A_0,A_1,\dotsc,A_{p-1}$ such that
 \[a\in A_j\Longleftrightarrow p^{a_m}qa\in [\frac{N}{p}j,\frac{N}{p}(j+1)).\]
 Then, all $A_j$ have the same cardinality, and if $d\mid p^{a_m-1}q$ (assuming $a_m>1$) with $A(\om^d)=0$, then $A_j(\om^d)=0$ for all $j$. Furthermore, $A_j(\om^{p^{a_m}q})\neq0$.
\end{prop}

\begin{proof}
 By Proposition \ref{PQ} we have $A(X^{p^{a_m}q})\equiv P(X^{p^{a_m}q})\Phi_p(X^{N/p})\bmod X^N-1$, and without loss of generality we have $\deg P(X^{p^{a_m}q})< N/p$.
 The hypothesis implies that 
 \[A_j(X^{p^{a_m}q})\equiv P(X^{p^{a_m}q})X^{\frac{N}{p}j}\bmod X^N-1,\]
 so that $A_j(1)=P(1)$ for all $j$, proving the first part.
 
 Now, let $d=p^kq^c$ with $A(\om^d)=0$, where $k<a_m$ and $c=0$ or $1$. We will use the language of multi-sets for this part; $d\cdot A$ is a union of $p$- and $q$-cycles, as
 multi-sets. We will show that every such cycle must belong exclusively to one of the mutli-sets $d\cdot A_j$, hence each $d\cdot A_j$ is also a union of $p$- and
 $q$-cycles, leading to $A_j(\om^d)=0$, for all $j$.
 
 Indeed, suppose that $da$ is a part of such a cycle; the other member of said cycle are $da+lN/r$, where $r=p$ or $q$, and $0\leq l\leq r-1$. Define $d'$ by $dd'=p^{a_m}q$. Then,
 for every $l$,
 \[d'da\equiv d'(da+lN/r)\bmod N,\]
 since $d'l\equiv0\bmod r$. This is true in case where $r=p$, because $p\mid d'$. If $r=q$, we can have $q$-cycles only if $c=0$, or equivalently $q\mid d'$. The above congruence
 clearly shows that any such cycle belongs to one of the multi-sets $d\cdot A_j$ (we remark that these multi-sets are mutually exclusive).
 
 For the last part, we just note that $P(\om^{p^{a_m}q})\neq0$, otherwise $\Phi_p(X^{N/p})$ would be a factor of $P(X^{p^{a_m}q})$, an impossibility, since the degree of the latter does not
 exceed $N/p$.
\end{proof}

\begin{prop}\label{induction}
 Let $A\ssq\ZZ_N$ be spectral with $\set{a: A(\om^{p^am})=0}=\set{a_1,\dotsc,a_m}$ and $B$ a spectrum. Consider the partition of $A$ into $A_0,A_1,\dotsc,A_{p-1}$ as in Proposition
 \ref{partition}. Suppose that the maximal $a\notin\set{a_1,\dotsc,a_m}$  with $0\leq a\leq n$ and $(B-B)\cap p^a\ZZ_N^{\star}\neq\vn$ (if it exists!) satisfies $a<a_m$. Then,
 each $A_j$ is spectral (if such $a$ does not exist, we have the same conclusion).
\end{prop}

\begin{proof}
 Let $B^i$ be the subset of elements $B$ whose $p$-adic digit at $a_m$ is equal to $i$. By hypothesis and Proposition \ref{partition},
 \[(B^i-B^i)\ssq(B-B)\sm (p^{a_m}\ZZ_N^{\star}\cup p^{a_m}q\ZZ_N^{\star})\ssq\set{d:A(\om^d)=0}\sm p^{a_m}\ZZ_N\ssq {d:A_j(\om^d)=0}.\]
 By pigeonhole principle, we may select one $B^i$ such that $\#B^i\geq \frac{1}{p}\#B=A_j(1)$. This can only be possible if we have equality, thus showing that each $A_j$ is spectral
 by Theorem \ref{diff}, having the same spectrum (actually, any $B^i$ would serve as such).
\end{proof}

If $A(1)=p^mq$ and $A(X)$ has at least $m$ roots of the form $\om^d$, where $d$ is a power of $p$, then $A$ has a special structure.

\begin{prop}\label{ManyRoots}
Let $A\ssq\ZZ_N$ with $A(1)=p^mq$. Suppose that $A(\om^{d_i})=0$, where $d_i$ are increasing powers of $p$, $1\leq i\leq m$ and $d_i\leq p^n$.
 Then either
 \[A(X^{d_j})\equiv P_{d_j}(X^{d_j})\Php\bmod X^N-1\]
 for all $j$, where $P_{d_j}\not\equiv0$ have nonnegative coefficients,
 or
 \[A(X^{d_m})\equiv 
 Q_{d_m}(X^{d_m})\Phq\bmod X^N-1,\]
 where $Q_{d_m}\not\equiv0$ has nonnegative coefficients.
\end{prop}

\begin{proof}
 By Lemma \ref{pqcycles} we obtain
 \[A(X^{d_j})\equiv P_{d_j}(X^{d_j})\Php+Q_{d_j}(X^{d_j})\Phq\bmod X^N-1,\]
 while on the other hand we can show inductively
 \[A(X^{d_{j+1}})\equiv p^jP_{d_1}^j(X^{d_{j+1}})\Php+\sum_{i=0}^j p^iQ_{d_1}^i(X^{d_{j+1}})\Phq\bmod X^N-1,\]
 for all $j$, where all the polynomials appearing in these two formulae have nonnegative coefficients, and satisfy the reccurence relations 
 \[P_{d_1}^j(X^{d_{j+2}})\equiv P_{d_1}^{j+1}(X^{d_{j+2}})\Php+Q_{d_1}^{j+1}(X^{d_{j+2}})\Phq\bmod X^N-1.\]
 Without loss of generality we can write
 \[P_{d_j}(X^{d_j})\equiv p^{j-1}P_{d_1}^{j-1}(X^{d_j}), Q_{d_j}(X^{d_j})\equiv \sum_{i=0}^{j-1} p^iQ_{d_1}^i(X^{d_j}),\]
 for all $j$. Putting $j=m-1$ and $X=1$ we get
 \[p^mq=A(1)=p^mP_{d_1}^{m-1}(1)+q\sum_{i=0}^{m-1} p^iQ_{d_1}^i(1),\]
 so by Lemma \ref{combo} we have either
 \begin{equation}\label{one}
 P_{d_1}^{m-1}(1)=q\	\text{ and }\	\sum_{i=0}^{m-1} p^iQ_{d_1}^i(1)=0, 
 \end{equation} 
 or 
 \begin{equation}\label{two}
 P_{d_1}^{m-1}(1)=0\	\text{ and }\	\sum_{i=0}^{m-1} p^iQ_{d_1}^i(1)=p^m. 
 \end{equation} 
 If \eqref{one} holds, then we have $Q_{d_1}^i\equiv0$ for all $i$, hence $Q_{d_j}\equiv0$ for all $j$, and $A(X^{d_j})\equiv P_{d_j}(X^{d_j})\Php$. Otherwise, if \eqref{two} holds, 
 then $P_{d_1}^{m-1}\equiv0$, so in this case we get $A(X^{d_m})\equiv Q_{d_m}(X^{d_m})\Phq$, as desired.
\end{proof}

Now, we can proceed with the conclusion of the Spectral$\Rightarrow$Tile proof. If $p^n\in D$, {\bf (T2)} holds vacuously, so we need only prove {\bf (T1)}, namely $A(1)=p^m$.
Suppose on the contrary that $A(1)>p^m$, hence by Proposition \ref{main}(1) we have $A(1)=p^mq$. If a spectrum $B$ satisfies the hypothesis of Proposition \ref{induction}, then
each $A_j$ is spectral, so by induction they satisfy {\bf (T1)}. Since $A_j(1)=p^{m-1}q$, we must have $A_j(\om^{p^n})=0$ for all $j$, yielding $A(\om^{p^n})=0$, contradicting
our assumption that $p^n\in D$. If the maximal $a$ such that $(B-B)\cap p^a\ZZ_N^{\star}\neq\vn$ satisfies $a>a_m$, then by Proposition \ref{main}(5) we get that
$A(\om^{p^{a_i}})=0$ for all $1\leq i\leq m$. Taking $\set{d_1,\dotsc,d_m}=\set{p^{a_1},\dotsc,p^{a_{m-1}},p^a}$ and applying Proposition \ref{ManyRoots} we get either
\[A(X^{d_m})\equiv P_{d_m}(X^{d_m})\Php\bmod X^N-1\]
or
\[A(X^{d_m})\equiv 
 Q_{d_m}(X^{d_m})\Phq\bmod X^N-1,\]
 where in each case $P_{d_m},Q_{d_m}\not\equiv0$ have nonnegative coefficients. However, this contradicts Proposition \ref{PQ}, since $A(\om^{p^n})A(\om^{p^aq})\neq0$. We conclude that
 $A$ must satisfy {\bf (T1)} as well, thus tiling $\ZZ_N$ by translations due to Theorem \ref{t1t2}.
 
 Lastly, we suppose that $p^n\notin D$, so that $A(\om^{p^n})=0$. In this case, $q\mid A(1)$ by Proposition \ref{lower}, hence $A(1)>p^m$ and $A(1)=p^mq$ by Proposition \ref{main}(1).
 Therefore, {\bf (T1)} holds and it remains to prove {\bf (T2)}, namely $A(\om^{p^{a_i}})=0$ for all $i$. If the maximal $a$ for which $(B-B)\cap p^a\ZZ_N^{\star}\neq\vn$ holds,
 satisfies $a>a_m$, for a
 spectrum $B$, then by applying Proposition \ref{main}(5) we deduce that {\bf (T2)} holds. Otherwise, $B$ satisfies the conditions of Proposition \ref{induction}. Therefore, each
 $A_j$ is spectral and satisfies {\bf (T2)}, yielding $A_j(\om^{p^{a_i}})=0$ for all $j$ and $1\leq i\leq m-1$, hence $A(\om^{p^{a_i}})=0$ for $1\leq i\leq m-1$. It remains to show
 that $A(\om^{p^{a_m}})=0$. We remark that an $a$ as described in Proposition \ref{induction} \emph{actually exists} in this case due to Proposition \ref{main}(4), and $a<a_m$.
 We apply Proposition \ref{ManyRoots} for $\set{d_1,\dotsc,d_m}=\set{p^{a_1},\dotsc,p^{a_{m-1}},p^a}$; if
 \[A(X^{d_j})\equiv P_{d_j}(X^{d_j})\Php\bmod X^N-1\]
 for all $j$, where $P_{d_j}\not\equiv0$ have nonnegative coefficients, then $A(\om^{p^aq})=0$, a contradiction. Therefore,
 \[A(X^{d_m})\equiv 
 Q_{d_m}(X^{d_m})\Phq\bmod X^N-1,\]
 where $Q_{d_m}\not\equiv0$ has nonnegative coefficients. Substituting $X$ by $X^{p^{a_m}/d_m}$, we get 
 \[A(X^{p^{a_m}})\equiv Q_{d_m}(X^{p^{a_m}})\Phq,\]
 yielding $A(\om^{p^{a_m}})=0$, completing the proof.






\section*{Acknowledgments} 
The authors are grateful to the anonymous reviewer for their valuable remarks.


\bibliographystyle{amsplain}


\begin{dajauthors}
\begin{authorinfo}[RDM]
  Romanos Diogenes Malikiosis\\
  Technische Universit\"at Berlin, Institut f\"ur Mathematik\\
  Sekretariat MA 4-1\\
  Stra{\ss}e des 17. Juni 136\\
  D-10623 Berlin, Germany\\
  malikios\imageat{}math\imagedot{}tu-berlin\imagedot{}de \\
  \url{https://sites.google.com/site/romanosdiogenesmalikiosis/}
\end{authorinfo}
\begin{authorinfo}[MNK]
  Mihail N. Kolountzakis\\
  Department of Mathematics and Applied Mathematics, University of Crete\\
  Voutes Campus\\
  700 13 Heraklion, Greece\\
  kolount\imageat{}gmail\imagedot{}com \\
  \url{http://mk.eigen-space.org/}
\end{authorinfo}
\end{dajauthors}

\end{document}